%% file: main.tex
\documentclass[a4paper,10pt]{amsart}
 \usepackage{amsfonts,mathrsfs}
 \usepackage{amsrefs}
  \usepackage{enumitem}
 \usepackage{amsthm}
 \usepackage{amssymb}
 \usepackage{tikz-cd}

\theoremstyle{definition}
\newtheorem{Def}{Definition}[section]
\newtheorem{ex}[Def]{Example}
\newtheorem{rem}[Def]{Remark}

\theoremstyle{plain}
\newtheorem{prop}[Def]{Proposition}
\newtheorem{thm}[Def]{Theorem}
\newtheorem{lem}[Def]{Lemma}
\newtheorem{cor}[Def]{Corollary}

\input{operators}
\input{fonts}

\title{Non-negative polynomials on generalized elliptic curves}
\author{Mario Kummer}
\address{Technische Universit\"at, Dresden, Germany} 
\email{mario.kummer@tu-dresden.de}
\author{Alja\v z Zalar}
\address{
Faculty of Computer and Information Science, University of Ljubljana  \& 
Faculty of Mathematics and Physics, University of Ljubljana  \&
Institute of Mathematics, Physics and Mechanics, Ljubljana, Slovenia.}
\email{aljaz.zalar@fri.uni-lj.si}
\thanks{M. Kummer was partially supported by DFG grant 502861109.
A. Zalar was supported by the Slovenian Research Agency 
program P1-0288 and grants J1-50002, J1-60011.}

\begin{document}

\subjclass[2020]{Primary 14P05, 14H52; Secondary 52A20}

\begin{abstract}
 We study the cone of non-negative polynomials on generalized elliptic curves. We show that the zero set of every extreme ray has dense real points.
 If a generalized elliptic curve is embedded via a complete linear system, then we show that the convex hull of its real points (taken inside any affine chart containing all real points) is a spectrahedron. On the way, we generalize a result by Geyer--Martens on $2$-torsion points in the Picard group of smooth real curves (of arbitrary genus) to possibly singular and reducible ones.
 \end{abstract}
\maketitle

\section{Introduction}
Given a projective variety $X$ over $\R$, i.e., a reduced and projective scheme of finite type over $\R$, and a line bundle $L$ on $X$, we consider the cone $\Pos(X,L)$ of global sections $f\in H^0(X,L\otimes L)$ that are non-negative on $X(\R)$.  If $L$ is very ample and the corresponding embedding of $X$ to projective space is projectively normal, then $\Pos(X,L)$ corresponds to the cone of quadrics that are non-negative on $X$, and this cone has been studied extensively for various different varieties $X$. For example, in the case of irreducible varieties of minimal degree \cite{blekherman1} and, more generally, varieties of Castelnuovo--Mumford regularity $2$ \cite{dream}, it admits the  description as all quadrics that can be written as a sum of squares of linear forms, generalizing a classical theorem by Hilbert \cite{hilbert}.
In this note, we examine $\Pos(X,L)$ when $X$ is a \emph{generalized elliptic curve} over $\R$, i.e., a one dimensional, projective, and connected variety over $\R$ with dense real points such that $\omega^\circ_X\cong\cO_X$. The prototype of a generalized elliptic curve is a reduced plane cubic. We will see that, even though not every element of $\Pos(X,L)$ is a sum of squares of elements of $H^0(X,L)$, the cone $\Pos(X,L)$ still admits a description with several, from the convex algebro-geometric point of view desirable properties. 
Our results rely on a description of extreme rays:

\begin{thm}\label{thm:extreme1}
    Let $X$ be a generalized elliptic curve over $\R$ and let $L$ be a line bundle on $X$.
If $f\in H^0(X,L\otimes L)$ generates an extreme ray of $\Pos(X,L)$, then the real points are Zariski dense in the zero set of $f$ on $X$.
\end{thm}
While in the smooth case Theorem \ref{thm:extreme1} is a direct consequence of the Riemann--Roch theorem, as observed for very ample $L$ in \cite{baldi}, the singular case is considerably more difficult.
A consequence of Theorem \ref{thm:extreme1} is that every non-negative global section on a generalized elliptic curve is a sum of squares of global sections of a not necessarily invertible sheaf:

\begin{thm}\label{thm:sumofsquares}
    Let $X$ be a generalized elliptic curve over $\R$ and let $L$ be a line bundle on $X$. There exists a coherent sheaf $\cG$ on $X$ together with a morphism $\varphi\colon \cG\otimes\cG\to L\otimes L$ such that the following holds:
    \begin{enumerate}
        \item The morphism $\varphi$ is \emph{positive semi-definite} in the sense that for every $g\in H^0(X,\cG)$ the global section $\varphi(g\otimes g)$ of $L\otimes L$ is non-negative on $X(\R)$.
        \item Every extreme ray of $\Pos(X,L)$ is generated by an element of the form $\varphi(g\otimes g)$ for some $g\in H^0(X,\cG)$.
    \end{enumerate}
    In particular  $\Pos(X,L)$ is the cone of sums of squares of global sections of $\cG$:
    \begin{equation*}
        \Pos(X,L)=\{\varphi(g_1\otimes g_1+\cdots+g_r\otimes g_r)\mid r\in\N \textrm{ and }g_1,\ldots,g_r\in H^0(X,\cG)\}.
    \end{equation*}
\end{thm}

A direct consequence of Theorem \ref{thm:sumofsquares} is that the dual cone $\Pos(X,L)^\vee$ can be described as  the cone of all linear forms $\ell\colon H^0(X,L\otimes L)\to\R$ such that 
\begin{equation*}
    B_\ell\colon H^0(X,\cG)\times H^0(X,\cG)\to\R,\, (g_1,g_2)\mapsto \ell(\varphi(g_1\otimes g_2))
\end{equation*}
is positive semi-definite. From this we obtain the following insight on convex hulls of generalized elliptic curves which appears to be new even in the smooth case.

\begin{cor}\label{cor:convexhull}
    Let $X\subseteq\pp^n$ be a generalized elliptic curve over $\R$ embedded via a complete linear system and let $H\subseteq\pp^n$ be a hyperplane with $H\cap X(\R)=\emptyset$. Then the convex hull of $X(\R)$ in $\R^n=(\pp^n\smallsetminus H)(\R)$ is a spectrahedron.
\end{cor}
Spectrahedra are convex semi-algebraic sets defined by a linear matrix inequality. They are fundamental objects in convex algebraic geometry; they are the feasible sets of semi-definite programming and have many desirable properties \cites{clausbook,timdaniel,siambook}.
Note that, while the convex hull of a curve is always a spectrahedral shadow \cites{scheiderercurves,scheiderercurves2}, i.e., the image of a spectrahedron under a linear map, 
it is rarely a spectrahedron itself. 
A more careful analysis of the proof of Corollar \ref{cor:convexhull} (see Remark \ref{rem:blocksize2}) will show that the description of the convex hull of $X(\R)$ in Corollary \ref{cor:convexhull} as a spectrahedron is particularly nice in that it is given by block matrices with blocks of size at most $\frac{n+1}{2}$.
Convex hulls of elliptic curves were studied before in \cites{scheiderergenusone,kummersinn}.

In order to deduce Corollary \ref{cor:convexhull} from Theorem \ref{thm:sumofsquares}, we need a certain divisibility result on the Picard group of a generalized elliptic curve. In Section \ref{sec:realcurves} we prove such for arbitrary real curves. From this we also deduce a generalization to singular curves of a result by Geyer--Martens \cite{geyermartens} on $2$-torsion points in the Picard group of a smooth real curve, which might be of independent interest (Corollar \ref{cor:geyer}).

\section{Preliminaries}
\subsection{Notation}
By a \emph{variety} over a field $k$ we mean a reduced and separated scheme of finite type over $k$, not necessarily irreducible. If $L$ is a line bundle on a normal irreducible variety $X$ and $f$ a rational section that does not vanish identically on an irreducible component of $X$, then we denote by 
\begin{equation*}
 \divv(f)   = \sum_{x\in X} v_x(f)\cdot x
\end{equation*}
its divisor (a Weil divisor); here the sum is taken over all points $x$ of codimension $1$ and $v_x(f)$ denotes the valuation of $f$ at $x$ in a trivialization of $L$ around $x$.
A \emph{curve} over $k$ is a projective variety of pure dimension $1$ over $k$. 
Sometimes we will use the fact that every line bundle $L$ on a curve $X$ over $k$ is isomorphic to $\cO_X(D)$ where $D$ is a (Weil) divisor supported on the regular points of $X$ \cite{stacks-project}*{Lemma 0AYM}. Here $\cO_X(D)$ denotes the subsheaf of the sheaf of total quotient rings of $\cO_X$ associated to $D$. Further, we say that a line bundle $L$ on a curve $X$ over $k$ has degree $0$ if the pull-back to every irreducible component of the normalization of $X$ has degree $0$. The subgroup of $\Pic(X)$ of degree $0$ line bundles on $X$ is denoted by $\Pic^\circ(X)$.

\subsection{Non-negativity of sections}
Let $X$ be a variety over $\R$ and let $L$ be a line bundle on $X$. We say that a section $f\in H^0(X,L\otimes L)$ is \emph{non-negative} or \emph{positive} at $x\in X(\R)$ if for a trivialization of $L|_U\cong \cO_X|_U$ on an open affine neighborhood $U$ of $x$ the section $f$ is mapped by the induced map
\begin{equation*}
    H^0(X,L\otimes L)\to L(U)\otimes L(U)\to\cO_X(U)\otimes\cO_X(U)\to \cO_X(U)
\end{equation*}
to a regular function on $U$ that is non-negative or positive at $x$, respectively. Note that the sign at $x$ of this regular function does not depend on the chosen trivialization of $L$. 
We denote by $\Pos(X,L)$ the closed convex cone of all $f\in H^0(X,L\otimes L)$ that are non-negative on $X(\R)$. 

We recall some straightforward observations that we will use in the proofs.
If $X$ is a curve and $f\in H^0(X,L\otimes L)$ is non-negative in a (Euclidean) neighborhood of a regular point  $x\in X(\R)$, then $v_x(f)$ is even. Moreover, if $g\in H^0(X,L\otimes L)$ is another section which satisfies $v_x(g)\geq v_x(f)$, then, for sufficiently small $\epsilon>0$, also $f\pm\epsilon g$ is non-negative in a neighborhood of $x$. If $\pi\colon \tilde{X}\to X$ is the normalization of $X$, then a section $h\in H^0(X,L\otimes L)$ is non-negative on the set of non-isolated real points if and only if its pull-back $\pi^*h$ is non-negative on $\tilde{X}(\R)$. If $L=\cO_X(D)$ for a divisor $D$ supported on regular points of $X$, then $f\in H^0(X,\cO_X(2D))$ is {non-negative} or {positive} at a real point $x$ of $U=X\smallsetminus\Supp(D)$ if and only if $f$ is {non-negative} or {positive} at $x$ as a regular function on $U$. In particular, if $f$ is non-negative on the real points of $U$ as regular function on $U$, then $f\in\Pos(X,L)$ because isolated real points are singular points of $X$ and thus in $U$.
\subsection{The Picard group of real curves}\label{sec:realcurves} In this section we prove some basic facts about the Picard group of not necessarily irreducible or smooth real curves. We start with the following well-known lemma over the complex numbers.
\begin{lem}\label{lem:complexdivisible}
    If $X$ is a curve over $\C$, then $\Pic^\circ(X)$ is divisible.
\end{lem}
\begin{proof}
    It is shown in the proof of \cite{stacks-project}*{Proposition 0C20} that $\Pic^\circ(X)$ can be obtained from $\Pic^\circ(\tilde{X})$, where $\tilde{X}$ is the normalization of $X$, by a sequence of extensions by $(\C^\times,\cdot)$ and $(\C,+)$. Now the claim follows because the groups $\Pic^\circ(\tilde{X})$, $(\C^\times,\cdot)$ and $(\C,+)$ are all divisible.
\end{proof}
 The following can be interpreted as a divisibility result on the so-called \emph{narrow class group} of a real curve. In the smooth case, it is  due to Hanselka \cite{hanselka}*{Section 4}.
 \begin{prop}\label{prop:narrow}
     Let $X$ be a curve over $\R$ such that $X(\R)$ is Zariski dense in $X$ and let $D$ be a divisor supported on non-real regular points of $X$. There exists a divisor $E$, supported on regular points, and a unit $f$ of the total quotient ring of $X$, which is non-negative on every real point where it is defined, such that
     \begin{equation*}
         D=2E+\divv(f).
     \end{equation*}
 \end{prop}
 \begin{proof}
     As the support of $D$ does not contain real points, we can write its pullback to $X_\C$ as $F+\bar{F}$ for some divisor $F$ where $\bar{F}$ denotes its complex conjugate. Let $F_1$ be a divisor supported on regular points fixed by the complex conjugation such that $\cO_{X_\C}(F-F_1)\in\Pic^\circ(X_\C)$. Since $\Pic^\circ(X_\C)$ is $2$-divisible by Lemma \ref{lem:complexdivisible}, there exists a divisor $G$ supported on regular points of $X_\C$ and a unit $g$ of the total quotient ring of $X_\C$ such that
     $2G+\divv(g)=F-F_1$.
     This implies that
     \begin{equation*}
         2(G+\bar{G})+\divv(g\bar{g})=F+\bar{F}-2F_1
     \end{equation*}
     and thus $D=2E+\divv(f)$ where $E$ is the divisor and $f$ the rational function whose pullback to $X_\C$ is $G+\bar{G}+F_1$ and $g\bar{g}$, respectively. It is clear that $f$ is non-negative at every real point where it is defined.
 \end{proof}
 Proposition \ref{prop:narrow} can be used to generalize a result by Geyer--Martens \cite{geyermartens} on the structure of $2$-torsion points on a smooth real curve to singular curves. 
 Namely, for a connected curve $X$ over $\R$ we consider the map
 \begin{equation*}
{\rm sg}\colon\Pic(X)_2\to \tilde{H}^0(X(\R),\{\pm1\})   
 \end{equation*}
 which sends a $2$-torsion point represented by a divisor $D$, supported on regular points, to the signs that a rational function, whose divisor is $2D$, takes on the different connected components of $X(\R)$. Here $\tilde{H}^0(X(\R),\{\pm1\})$ denotes the $0$th reduced singular cohomology group of $X(\R)$ with coefficients in the multiplicative group $\{\pm1\}$. It is straightforward to check that this is a well-defined group homomorphism. The elements of $\Pic(X)^+_2=\ker(\rm sg)$ are called \emph{positive $2$-torsion points} because the associated rational function can be chosen to be non-negative on $X(\R)$. The following was shown in the smooth case by Geyer--Martens in \cite{geyermartens}*{Section 5}.
 \begin{cor}\label{cor:geyer}
     For every connected curve $X$ over $\R$ the map $\rm sg$ is surjective.
 \end{cor}
\begin{proof}
    Because $X$ is projective, there exists an open affine subset $U$ of $X$ which contains $X(\R)$ and all singular points. Indeed, if $X$ is embedded to $\pp^n$, then, after applying a linear change of coordinates if necessary, we can take the principal open subset of all points where $x_0^2+\cdots+x_n^2$ is non-zero. Embedding $U$ to $\A^m$ gives an embedding of $X(\R)$ to $\R^m=\A^m(\R)$. Let $X_1$ be a union of connected components of $X(\R)$ and $X_2=X(\R)\smallsetminus X_1$. Since $X_1$ and $X_2$ are disjoint compact subsets of $\R^m$, there is a continuous function $\psi\colon\R^m\to\R$ which satisfies $\psi(x)\leq-1$ for $x\in X_1$ and $\psi(x)\geq1$ for $x\in X_2$. By the Stone--Weierstrass theorem we can 
     uniformly approximate $\psi$ on a compact subset of $\R^m$ containing $X(\R)$ arbitrarily close by a polynomial. In particular, there exists $g\in\R[x_1,\ldots,x_m]$ with $g(x)<0$ for $x\in X_1$ and $g(x)>0$ for $x\in X_2$. After replacing $g$ by a small perturbation if necessary, we can further assume that $g$ does not vanish on any singular point of $X$. We apply Proposition \ref{prop:narrow} to the principal divisor $D=\divv(g)$ and obtain a divisor $E$, supported on regular points, and a unit $f$ of the total quotient ring of $X$, which is non-negative on every real point where it is defined, such that $\divv(\frac{g}{f})=2E$. The $2$-torsion point $\cO_X(E)$ of $\Pic(X)$ is then mapped by $\rm sg$ to a cohomology class represented by the function which is $-1$ on $X_1$ and $+1$ on $X_2$. Since $X_1$ was chosen to be an arbitrary union of connected components of $X(\R)$ and $X_2$ its complement, this shows the claim.
\end{proof}
We will see that positive $2$-torsion points are closely related to those extreme rays of $\Pos(X,L)$ with only regular zeros. Corollary \ref{cor:geyer} gives a way to count them.

\subsection{Generalized elliptic curves}\label{sec:genell}
A \emph{generalized elliptic curve} over $\R$ is a connected curve over $\R$ with dense real points such that $\omega^\circ_X\cong\cO_X$. Here $\omega^\circ_X$ is the dualizing sheaf of $X$. The prototype of a generalized elliptic curve is a reduced plane cubic curve but there are more examples.

\begin{Def}
   For $n\in\N$ the \emph{N\'eron $n$-gon} is obtained by taking a copy $X_i$ of $\pp^1$ for each $i\in\Z/n\Z$ and gluing the point $\infty$ of $X_i$ to the point $0$ of $X_{i+1}$ such that the intersection points become (ordinary) nodes.
\end{Def} 

Note that for $n\leq 3$ the N\'eron $n$-gon is isomorphic to a planar cubic curve.
The following characterization is a classical result of Kodaira.

\begin{thm}[\cite{kodaira}*{Theorem 6.2}]\label{thm:kodairaclass}
    Every generalized elliptic curve is either isomorphic to a reduced plane cubic or to the N\'eron $n$-gon.    
\end{thm}

We list the resulting possible types of generalized elliptic curves over $\R$.

\begin{cor}\label{cor:list}
    Let $X$ be a generalized elliptic curve over $\R$. Then $X$ is isomorphic to one of the following: a non-singular elliptic curve, a N\'eron $n$-gon for $n\in\N$, a rational curve with one isolated node, a smooth plane conic and a line meeting in two non-real points, a rational curve with one cusp, a smooth plane conic and one of its tangents, or three planar lines that meet in one point.
\end{cor}

\begin{rem}\label{rem:inclusiontonorm}
    Let $X$ be a generalized elliptic curve over $\R$ and let $\pi\colon\tilde{X}\to X$ be its normalization. The goal of this remark is to describe the inclusion $\cO_X\hookrightarrow\pi_*\cO_{\tilde{X}}$ of coherent sheaves on $X$ in detail. At non-singular points, this is an isomorphism. If $x\in X(\C)$ is a point at which $m$ (complex) branches meet transversally, then there are $m$ distinct $\C$-points points $x_1,\ldots,x_m$ in $\tilde{X}(\C)$ lying over $x$ and the local ring $\cO_{X,x}$ consists of all $f\in\cO_{\tilde{X},\pi^{-1}(x)}$ such that $f(x_i)=f(x_j)$ for $i,j=1,\ldots,m$. If two branches are tangent to each other at $x$, then there are two distinct $\C$-points $x_1,x_2$ in $\tilde{X}(\C)$ lying over $x$ and the local ring $\cO_{X,x}$ consists of all $f\in\cO_{\tilde{X},\pi^{-1}(x)}$ such that $f(x_1)=f(x_2)$ and $f'(x_1)=f'(x_2)$ (the derivative is taken with respect to a suitable local parameter). Finally, if $x\in X(\R)$ is a cusp, then there is a unique $\R$-point $y$ in $\tilde{X}(\R)$ lying over $x$ and $\cO_{X,x}$ consists of all $f\in\cO_{\tilde{X},y}$ such that $f'(y)=0$. This covers all possibilities. 
\end{rem}

We conclude this section with an estimate of positive $2$-torsion points on generalized elliptic curves.
\begin{lem}\label{lem:posbound}
    Let $X$ be a generalized elliptic curve over $\R$. There are at most $2$ positive $2$-torsion points in $\Pic(X)$.
\end{lem}
\begin{proof}
    Unless $X$ is smooth, every irreducible component is rational. In this case, we have by \cite{stacks-project}*{Lemma 0CE6}  that $|\Pic(X)_2|\leq2$. If $X$ is smooth with $X(\R)$ having $r$ connected components, then
    \begin{equation*}
        r=\dim_{\F_2}(\Pic(X)_2)=\dim_{\F_2}(\Pic(X)^+_2)+r-1.
    \end{equation*}
    Here the first equality is a classical fact about real elliptic curves and the second equality holds
    by Corollary \ref{cor:geyer}. This proves that $\dim_{\F_2}(\Pic(X)^+_2)=1$.
\end{proof}
\subsection{Spectrahedra}
A \emph{spectrahedron} in $\R^n$ is the inverse image of the cone of positive semidefinite matrices under an affine linear map from $\R^n$ to the vector space of real symmetric matrices of size $d\times d$ for some $d\in\N$. Equivalently, a spectrahedron is a set of the form
\begin{equation*}
    \{x\in\R^n\mid A_0+x_1A_1+\cdots+x_nA_n\textrm{ is positive semi-definite}\},
\end{equation*}
where $A_0,\ldots,A_n$ are real symmetric matrices of the same size. A \emph{spectrahedral shadow} is the image of a spectrahedron under an affine linear map. For reading about the many interesting properties and applications of spectrahedra and their shadows we recommend the books \cites{clausbook,timdaniel,siambook}.
\section{Proofs and examples}
Let $X$ be a generalized elliptic curve over $\R$ and let $L$ be a line bundle on $X$.
We let $X=\cup_{i=1}^rX_i$ be the decomposition of $X$ into irreducible components and $\pi\colon\tilde{X}\to {X}$ the normalization of ${X}$. Hence $\tilde{X}$ is a smooth projective real curve with irreducible components $\tilde{X}_1,\ldots,\tilde{X}_r$ where each $\tilde{X}_i$ is the normalization of $X_i$. There exists a short exact sequence
\begin{equation}\label{eq:normsequence}
    0\to\cO_X\to\pi_*\cO_{\tilde{X}}\to\shS\to0
\end{equation}
of coherent sheaves on $X$ where $\shS$ is the cokernel of the map $\cO_X\to\pi_*\cO_{\tilde{X}}$.  The support of $\shS$ is the singular locus of $X$.  For a closed point $x\in X$ the dimension (as $\R$-vector space) of the stalk $\shS_x$ is denoted by $\delta(x)$.
Let $L$ be a line bundle on $X$ and $M=L\otimes L$. Tensoring (\ref{eq:normsequence}) by $M$ and passing to the long exact sequence of cohomology, we obtain
\begin{equation}\label{eq:longsequence}
  \begin{tikzcd}
   0\arrow{r}& H^0(X,M) \arrow{r}{\alpha}& \oplus_{i=1}^r H^0(\tilde{X}_i,M_i)\arrow{r}{\beta} & H^0(X,\shS)
  \end{tikzcd}
\end{equation}
where we denote by $M_i$ the pullback of $M$ to $\tilde{X}_i$. 
For $f\in H^0(X,M)$ and $i=1,\ldots,r$ we let $f_i\in H^0(\tilde{X}_i,M_i)$ be the pullback of $f$ to $\tilde{X}_i$ so that $\alpha(f)=f_1+\cdots+f_r$.
\begin{rem}\label{rem:upperbound}
The dimension of the image of $H^0(\tilde{X}_i,M_i)$ under $\beta$ is bounded from above by the sum $\delta_i=\sum_{x\in X_i}\delta(x)$. Using Remark \ref{rem:inclusiontonorm} and going through  Corollary \ref{cor:list} we see that this sum is at most $1$ for irreducible $X$ and that in general $\delta_i\leq2$ for every $i$. This property can also be deduced directly from the definitions by making use of the fact that $\pi^!_i\omega^\circ_X$, where $\pi_i\colon\tilde{X}_i\to X$ is the inclusion $\tilde{X}_i\hookrightarrow\tilde{X}$ composed with the normalization $\pi\colon\tilde{X}\to X$, is, on the one hand, the canonical sheaf of $\tilde{X}_i$ by \cite{hartshorne}*{Exercise III.7.2}, and, on the other hand, the codifferent sheaf $\pi^!_i\cO_X$ of $\pi_i$.
\end{rem}
\begin{lem}\label{lem:oneindep}
    Let $f\in H^0(X,M)$ and let $j\in\{1,\ldots,r\}$ such that $f_j\neq0$. Let $p,q\in \tilde{X}_j$ distinct points with $f_j(p)=f_j(q)=0$ such that neither $\pi(p)$ nor $\pi(q)$ lies on an irreducible component of ${X}$ on which $f$ vanishes identically. We consider the vector space $V$ of all $g\in H^0(X,M)$ that satisfy:
    \begin{enumerate}
        \item For all $i\neq j$ there exists $\lambda_i\in\R$ such that $g_i=\lambda_i\cdot f_i$.
        \item $\divv(g_j)\geq \divv(f_j)-p-q$.
    \end{enumerate}
    The vector space $V$ has dimension at least $2$.
\end{lem}

\begin{proof}
  If $X$ is smooth, then the claim follows from the Riemann--Roch theorem. Hence we assume from now on that $X$ is singular. Then each $\tilde{X}_i$ is isomorphic to $\pp^1$ and thus the space $W$ of all $h\in H^0(\tilde{X}_j,M_j)$ with $\divv(h)\geq \divv(f_j)-p-q$ has dimension $3$. By Remark \ref{rem:upperbound} the dimension of $\beta(H^0(\tilde{X}_j,M_j))$ is at most $2$. Since $\dim(W)=3$, there is $0\neq h\in W\cap\ker(\beta)$. Thus the claim follows if $h$ is linearly independent from $\alpha(f)$. Hence we assume in the following that $h$ and $\alpha(f)$ are linearly dependent. This can only happen if $f_i=0$ for all $i\neq j$ and thus $\beta(f_j)=0$. If $r>1$, then for $g\in H^0(\tilde{X}_j,M_j)$ the condition $\beta(g)=0$ is equivalent to a vanishing condition on the points on $\tilde{X}_j$ that are mapped by $\pi$ to some $X_i$ with $i\neq j$. Since $f_j$ satisfies this condition and since neither $p$ nor $q$ are mapped by $\pi$ to some $X_i$ with $i\neq j$, the condition $\divv(g)\geq \divv(f_j)-p-q$ implies that $g\in H^0(\tilde{X}_j,M_j)$ satisfies this vanishing condition as well. Thus $W$ lies in the kernel of $\beta$ which implies the claim.
  It remains to treat the case that $r=1$. Then by Remark \ref{rem:upperbound} the dimension of $\beta(H^0(\tilde{X}_j,M_j))$ is $1$ and thus the space $W\cap\ker(\beta)$ is at least of $2$-dimensional which implies the claim.
\end{proof}

\begin{proof}[Proof of Theorem \ref{thm:extreme1}]
    Let $f\in H^0(X,L\otimes L)$ be non-negative on $X(\R)$ and let $S$ be the set of indices $i\in\{1,\ldots,r\}$ for which $f$ vanishes identically on $X_i$. Assume that there is $x\in X\smallsetminus(\cup_{i\in S}X_i)$ with $f(x)=0$ such that $x\not\in X(\R)$. We have to show that $f$ does not generate an extreme ray of $\Pos(X,L)$.

    Let $X_j$ be an irreducible component of $X$ that contains $x$. Note that this implies $j\not\in S$.
    There exists $p\in\pi^{-1}(x)\cap \tilde{X}_j$ and we let $q=\bar{p}$ be its complex conjugate. Because $x\not\in X(\R)$, the points $p$ and $q$ are two distinct elements of $\tilde{X}_j$. Now $f$, $p$ and $q$ satisfy the assumptions of Lemma \ref{lem:oneindep}. Thus there exists $g\in H^0(X,L\otimes L)$ which is linearly independent of $f$ such that: 
    \begin{enumerate}
        \item For all $i\neq j$ there exists $\lambda_i\in\R$ such that $g_i=\lambda_i\cdot f_i$.
        \item $\divv(g_j)\geq \divv(f_j)-p-q$.
    \end{enumerate}
    Let $y\in\tilde{X}$, say $y\in\tilde{X}_i$, such that $\pi(y)\in X(\R)$. 
    Conditions (1) and (2) imply, by choice of $p$ and $q$, that $v_y({g}_i)\geq v_y({f}_i)$.
    Therefore, there exists $\epsilon_i>0$ such that ${f}_i\pm\epsilon{g}_i$ is non-negative on $\tilde{X}_i(\R)$. This implies that $f\pm\epsilon g$ is non-negative on the set of non-isolated points of $X(\R)$ for $\epsilon=\min_{i=1}^r\epsilon_i$. 

    Finally, we note that if $f$ vanishes at an isolated point $z$ of $X(\R)$, then $g$ vanishes this point as well. This follows from $v_y(g_i)\geq v_y(f_i)$ by choosing above $y$ in the preimage of $z$ under $\pi$.
    At an isolated real point where $f$ is positive, we can ensure that $f\pm\epsilon g$ is positive as well by replacing $\epsilon$ by a smaller positive number. Since there is at most one isolated point, there exists $\epsilon>0$ such that $f\pm\epsilon g$ is non-negative on $X(\R)$. This shows that $f$ is not an extreme ray of $\Pos(X,L)$.
\end{proof}

\begin{proof}[Proof of Theorem \ref{thm:sumofsquares}]
    Let $D$ be a divisor on $X$ supported on regular points such that $L\cong\cO_X(D)$, the subsheaf of the sheaf of total quotient rings of $\cO_X$ associated to $D$. Let $s\in H^0(X,L\otimes L)$ generate an extreme ray of $\Pos(X,L)$. Consider the partial normalization $\tilde{Y}\to X$ at all singular points of $X$ where $s$ vanishes. Further, let $Y\subseteq\tilde{Y}$ be the union of all irreducible components of $\tilde{Y}$ where the pullback of $s$ to $\tilde{Y}$ is not identically zero. Finally, let $\rho\colon Y\to X$ be the composition of the inclusion $Y\hookrightarrow\tilde{Y}$ with $\tilde{Y}\to X$. Then the restriction $D'=D|_Y$ satisfies $\rho^*L\cong\cO_Y(D')$. We choose $f\in\cO_Y(2D')$, a rational function on $Y$, which corresponds to $\rho^*s$ under $\rho^*L\cong\cO_Y(D')$. We denote by $E_1$ the restriction of the zero divisor of $\rho^*s$ to the set of all regular points of $Y$ that lie over a singular point of $X$. Then the principal divisor of $f$ is given by $\divv(f)=E_1+E_2-2D'$ for some effective divisor $E_2$ supported on regular points. By Theorem \ref{thm:extreme1} the divisor $E_2$ is a sum of real non-isolated points of $Y$. Because $f$ is non-negative on $Y(\R)$, it follows that $E_2=2F$ for some divisor $F$ because real zeros of non-negative sections must occur with even multiplicity. Consider the bilinear map
    \begin{equation*}
        \begin{tikzcd}
            \cO_Y(F)\times \cO_Y(F)\arrow{r}&\cO_Y(E_2)\arrow{r}{\cdot f}&\cO_Y(2D'-E_1)\subseteq\cO_X(2D)
        \end{tikzcd}
    \end{equation*}
    where the first arrow is the multiplication map. Letting $A=\cO_Y(F)$, the induced morphism $\psi\colon\rho_*A\otimes\rho_*A\to L\otimes L$ of coherent sheaves on $X$ is positive semi-definite and $s$ is the image of a square of a global section of $A$. We argue that, up to isomorphism and multiplication by a positive scalars (which does not destroy the desired properties), there are only finitely many choices for $\psi$ and $A$. Indeed, because there are only finitely many singular points, there are only finitely many choices for $\tilde{Y}$. Since there are only finitely many irreducible components, there are only finitely many choices for $Y$. As the degree of $L$ is finite, there are only finitely many possibilities for the divisor $E_1$. This implies that there are, up to linear equivalence, only finitely many possibilities for $E_2$. Since $2F=E_2$ and because the Picard group of $Y$ has only finitely many $2$-torsion points, see  \cite{stacks-project}*{Lemma 0CE6}, there are only finitely many possibilities for the divisor class of $F$ as well. Different resulting maps $\psi$ for the same $F$ can differ only by multiplication with positive scalars. Hence we arrive at finitely many coherent sheaves $A_1,\ldots,A_m$ on $X$ and positive semi-definite morphisms $\psi_i\colon A_i\otimes A_i\to L\otimes L$ such that every generator of an extreme ray of $\Pos(X,L)$ is the image of a square of a global section of one of the $A_i$. Hence we can choose $\cG$ and $\varphi$ to be the direct sums of the $A_i$ and $\psi_i$, respectively. The additional statement follows from the fact that every closed convex cone that is pointed is generated by its extreme rays.
\end{proof}

\begin{rem}
    The proof of Theorem \ref{thm:sumofsquares} simplifies when the extreme ray $s$ has only regular zeros. In this case, the proof shows that $s$ is the image of a square under a positive semi-definite isomorphism $\psi\colon A\otimes A\to L\otimes L$ where $A$ is a line bundle. Hence we have $A=L\otimes T$ where $T\in\Pic(X)_2$ is an positive $2$-torsion point. By Lemma \ref{lem:posbound} there are at most two such points. However, as Example \ref{ex:cusp1} shows, line bundles are not enough to describe extreme rays that vanish at singular points.
\end{rem}

\begin{ex}\label{ex:smooth}
    Consider the case that $X$ is a smooth elliptic curve over $\R$. A plane affine model of $X$ is given by the equation $y^2=x\cdot p(x)$ where $p$ is a univariate polynomial of degree $2$ with simple zeros which are either positive or non-real. 
    Let $P$ be the point $(0,0)$ and let $O$ be the point at infinity. Then the unique (see Lemma \ref{lem:posbound}) non-trivial positive $2$-torsion point is represented by the divisor $T=P-O$. Indeed, we have $\divv(x)=2T$ and $x$ is non-negative on $X(\R)$. Hence, if $D$ a divisor on $X$ and $L=\cO_X(D)$, then every extreme ray of $\Pos(X,L)$ is either a square of a global section of $\cO_X(D)$ or the image of a square under the map
    \begin{equation*}
        \begin{tikzcd}
            \cO_X(D+T)\otimes \cO_X(D+T)\arrow{r}&\cO_X(2D+2T)\arrow{r}{\cdot x}&\cO_X(2D).
        \end{tikzcd}
    \end{equation*}
\end{ex}

\begin{ex}\label{ex:cusp1}
    We illustrate the proof of Theorem \ref{thm:sumofsquares} in the case when $X$ is the rational curve with one cusp. A plane affine model of $X$ is given by the equation $y^2=x^3$. Let $O$ be the point at infinity and $L=\cO_X(dO)$ for $d\in\N$. Note that for $d=1,2$ the line bundle $L$ is ample but not very ample. The only possibilities for $\rho\colon Y\to X$ are the identity $X\to X$ and the normalization $\pp^1\to X$. In the first case, we have $E_1=0$. Since the Picard group of $X$ is the additive group $\R\times\Z$, which is torsion-free, the only possibility, up to a positive scalar, for $\psi$ is the identity on $L\otimes L$. In the case $Y=\pp^1$, we denote by $P$ the preimage of the cusp under $\pi$. Since $E_2$ has even degree and $E_1$ is linearly equivalent to $2dO-E_2$, we either have $E_1=2kP$ for $k=1,\ldots,d$. Here, by abuse of notation, we denote the preimage of $O$ under $\pi$ also by $O$. Again, as the Picard group of $\pp^1$ is torsion-free, we find that $F$ is linearly equivalent to $dO-kP\sim (d-k)O$.
    In order to describe the resulting map $\psi$ on global sections explicitly, we identify the global sections of $\cO_X(mO)$ and $\cO_Y(mO)$ with the spaces $\R[t^2,t^3]_{\leq m}$ and $\R[t]_{\leq m}$ of univariate polynomials in $\R[t^2,t^3]$ and $\R[t]$ that have degree at most $m$. We get the map
    \begin{equation*}
        \begin{tikzcd}
            \R[t]_{\leq d-k}\times \R[t]_{\leq d-k}\arrow{r}&\R[t]_{\leq 2d-2k}\arrow{r}{\cdot t^{2k}}&\R[t^2,t^3]_{\leq 2d}.
        \end{tikzcd}
    \end{equation*}
    The images of squares are thus of the form $(t^{k}g)^2$ where $g\in\R[t]_{\leq d-k}$. For $k\geq2$ we have $t^{k}g\in\R[t^2,t^3]$, so that these sections of $L\otimes L$ are actually squares of sections of $L$. Hence, these are already covered by the case $Y=X$. In conclusion, every extreme ray of $\Pos(X,L)$ is either the square of a section of $L=\cO_X(dO)$ or of $\pi_*\cO_{\pp^1}((d-1)O)$, and we can choose $\cG$ to be the sum of these two sheaves.
\end{ex}

Theorem \ref{thm:sumofsquares} gives the following nice description of the dual cone $\Pos(X,L)^\vee$.
\begin{cor}\label{cor:speccone}
    Let $X$ be a generalized elliptic curve over $\R$ and let $L$ be a line bundle on $X$. The dual cone $\Pos(X,L)^\vee$ is the a spectrahedral cone of all linear forms $\ell\colon H^0(X,L\otimes L)\to\R$ such that 
\begin{equation*}
    B_\ell\colon H^0(X,\cG)\times H^0(X,\cG)\to\R,\, (g_1,g_2)\mapsto \ell(\varphi(g_1\otimes g_2))
\end{equation*}
is positive semi-definite.
\end{cor}
\begin{rem}\label{rem:blocksizes}
    The proof of Theorem \ref{thm:sumofsquares} shows that the description of $\Pos(X,L)^\vee$ in Corollary \ref{cor:speccone} consists of several blocks, one corresponding to each of the sheaves $A_1,\ldots,A_m$ on $X$. A more careful analysis gives bounds on the size of these blocks. By the construction, the size of the block corresponding to $A_i$ is equal to $h^0(X,A_i)$. Using the notation as in the proof of Theorem \ref{thm:sumofsquares}, this number is equal to $h^0(Y,\cO_Y(F))$. By the Riemann--Roch theorem \cite{stacks-project}*{Lemma 0BS6} we have
    \begin{equation}\label{eq:riemannroch}
        h^0(Y,\cO_Y(F))=\deg(F)+h^0(Y,\cO_Y)-h^1(Y,\cO_Y)+h^0(Y,\cO_Y(K_Y-F))
    \end{equation}
    where $K_Y$ is a divisor on $Y$, supported on regular points, such that $\cO_Y(K_Y)\cong\omega^\circ_Y$. If $Y=X$, then Equation \ref{eq:riemannroch} simplifies to
    \begin{equation*}
        h^0(Y,\cO_Y(F))=\deg(D)+h^0(X,\cO_X(-F)).
    \end{equation*}
    If $L$ is ample, then divisor $F$ has positive degree and we obtain $\deg(D)$ as the size of the block. In the case that $Y$ is different from $X$, then $0\leq\deg(F)<\deg(D)$ and $\deg(K_Y)<0$.  Hence Equation \ref{eq:riemannroch} implies 
    \begin{equation*}
        h^0(Y,\cO_Y(F))\leq\deg(D)-1+h^0(Y,\cO_Y).
    \end{equation*}
    We can further assume that $Y$ is connected because otherwise $\cO_Y(F)$ is the direct sum of its restrictions to the connected components of $Y$ and our block decomposes into smaller blocks. In this case $h^0(Y,\cO_Y)=1$ and we find $h^0(Y,\cO_Y(F))\leq\deg(D)$. For later reference we note that if $L$ is ample, the Riemann--Roch theorem implies
    \begin{equation*}
        h^0(X,\cO_X(2D))=2\deg(D)
    \end{equation*}
    so that every block is of size at most $\frac{1}{2}h^0(X,L\otimes L)$.
\end{rem}
The proof of Corollary \ref{cor:convexhull}  is a combination of Corollary \ref{cor:speccone} and Proposition \ref{prop:narrow}.
\begin{proof}[Proof of Corollary \ref{cor:convexhull}]
    Because $X(\R)\subseteq\R^n=(\pp^n\smallsetminus H)(\R)$ is compact, there exists an affine linear functional $\R^n\to\R$ which is positive on $X(\R)$. Let $\ell$ be the corresponding global section of $M=\cO_X(1)$. By slightly perturbing $\ell$ if necessary, we can assume that the zero set of $\ell$ consists only of regular points of $X$. Let $D$ be its zero divisor. By Proposition \ref{prop:narrow} there exists a divisor $E$, supported on regular points, and a unit $f$ of the total quotient ring of $X$, which is non-negative on every real point where it is defined, such that
     $D=2E+\divv(f)$. We have isomorphisms
     \begin{equation*}
        \begin{tikzcd}
            \cO_X(E)\otimes \cO_X(E)=\cO_X(2E)\arrow{r}{\cdot f}&\cO_X(D)\arrow{r}{\cdot \ell}&M.
        \end{tikzcd}
    \end{equation*}
    Letting $L=\cO_X(E)$, the induced vector space isomorphism $$\psi\colon H^0(X,L\otimes L)\to H^0(X,M)={\rm Aff}(\R^n)$$ to the space ${\rm Aff}(\R^n)$ of affine linear functionals on $\R^n$ identifies $\Pos(X,L)$ with the cone $K$ of affine linear functionals that are non-negative on $X(\R)$. On the other hand, the affine linear map 
    \begin{equation*}
        {\rm ev}\colon\R^n\to{\rm Aff}(\R^n)^\vee,
    \end{equation*}
    that maps a point to the corresponding point evaluation, satisfies ${\rm ev}^{-1}(K^\vee)=\conv(X(\R))\subseteq\R^n$. Therefore, the set $\conv(X(\R))={\rm ev}^{-1}(\psi^\vee(\Pos(X,L)^\vee))$ is a spectrahedron by Corollary \ref{cor:speccone}.
\end{proof}
\begin{rem}\label{rem:blocksize2}
    It follows from Remark \ref{rem:blocksizes} that, in the situation of Corollary \ref{cor:convexhull}, the description of $\conv(X(\R))\subseteq\R^n$ consists of blocks, each of size at most $\frac{n+1}{2}$. Note that in general one cannot expect a description that only involves blocks of smaller size, even when allowing a description as a spectrahedral shadow \cite{scheiderercurves2}.
\end{rem}
\begin{ex}\label{ex:smooth2}
 Consider the curve $X$ in $\pp^3$ defined by
 \begin{align*}
     x_0^2-x_1^2-x_2^2&=0,\\
     x_0^2-4x_1^2-x_3^2&=0.
 \end{align*}
 One checks that $X$ is a smooth elliptic curve. The hyperplane defined by $x_0=0$ at infinity does not intersect $X(\R)$. Thus Corollary \ref{cor:convexhull} says that the convex hull of $X(\R)$ in the affine chart $x_0\neq0$ is a spectrahedron. Indeed, it is the set of all $(x_1,x_2,x_3)$ such that the matrix
 \begin{equation*}
     \begin{pmatrix}
        1-x_1&x_2&0&0\\
        x_2&1+x_1&0&0\\
        0&0&1-2x_1&x_3\\
        0&0&x_3&1+2x_1
     \end{pmatrix}
 \end{equation*}
 is positive semi-definite. The two blocks correspond to line bundles $L_1$ and $L_2$ with $L_i\otimes L_i\cong\cO_X(1)$ which differ by the non-trivial positive $2$-torsion point of $\Pic(X)$.
\end{ex}
\begin{ex}\label{ex:cusp2}
 Consider the curve $X$ in $\pp^3$ defined by
 \begin{align*}
     x_0x_1-x_1^2-x_2^2&=0,\\
     x_0x_2-x_1x_2-x_3^2&=0.
 \end{align*}
 It has a cusp at $(1:1:0:0)$ and can be parametrized as follows
 \begin{equation*}
     \pp^1\to X,\, (s:t)\mapsto(s^4+t^4:s^4:s^2t^2:st^3).
 \end{equation*}
 Therefore, it is isomorphic to the curve from Example \ref{ex:cusp1} and $\cO_X(1)$ is isomorphic to $\cO_X(4O)$ where $O=(1:0:0:0)$. We use the notation as in Example \ref{ex:cusp1}.
 The hyperplane defined by $x_0=0$ at infinity does not intersect $X(\R)$. Thus Corollary \ref{cor:convexhull} says that the convex hull of $X(\R)$ in the affine chart $x_0\neq0$ is a spectrahedron. Indeed, it is the set of all $(x_1,x_2,x_3)$ such that the matrix
 \begin{equation*}
     \begin{pmatrix}
        1-x_1&x_2&0&0\\
        x_2&x_1&0&0\\
        0&0&1-x_1&x_3\\
        0&0&x_3&x_2
     \end{pmatrix}
 \end{equation*}
 is positive semi-definite. The first block corresponds to the line bundle $\cO_X(2O)$ and the second one to the (non-invertible) sheaf $\pi_*\cO_{\pp^1}(O)$. The determinant of the second block cuts out the cone over $X$ whose apex is the cusp of $X$.
\end{ex}
\begin{figure}
    \centering
    \includegraphics[width=0.4\linewidth]{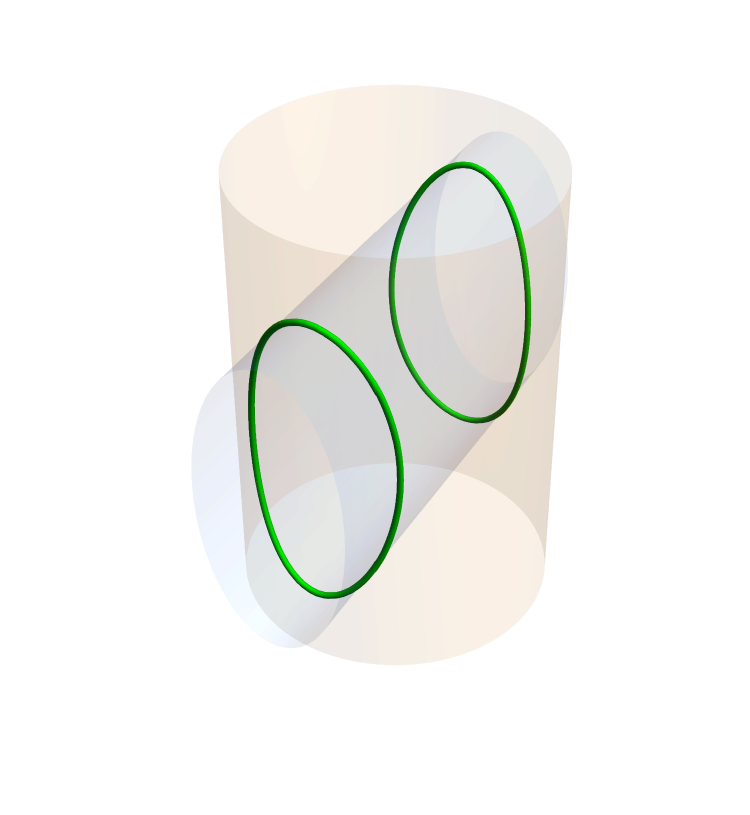}\hspace{1cm} \includegraphics[width=0.4\linewidth]{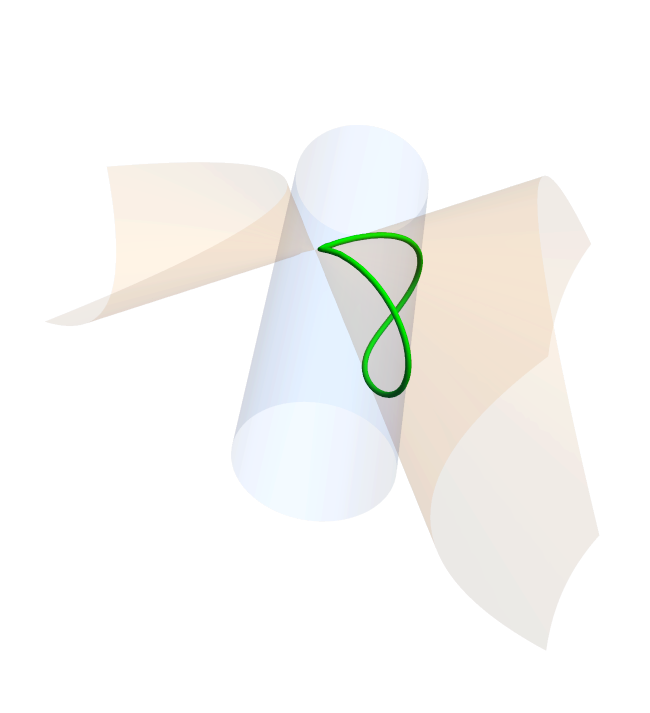} 
    \caption{The (generalized) elliptic curves from Example \ref{ex:smooth2} (left) and Example \ref{ex:cusp2} (right).}
    \label{fig:enter-label}
\end{figure}
\bigskip

\bibliographystyle{alpha}
\bibliography{biblio}
 \end{document}

%% file: operators.tex
\usepackage{mathrsfs}

\newcommand{\Pos}{\operatorname{P}}

\newcommand{\Pic}{\operatorname{Pic}}

\newcommand{\conv}{\operatorname{conv}}

\newcommand{\Supp}{\operatorname{Supp}}

\newcommand{\divv}{\operatorname{div}}



%% file: fonts.tex
\newcommand{\shS}{{\mathscr S}}

\newcommand{\cG}{{\mathcal G}}

\newcommand{\cO}{{\mathcal O}}

\newcommand{\A}{{\mathbb A}}

\newcommand{\C}{{\mathbb C}}
\newcommand{\R}{{\mathbb R}}
\newcommand{\pp}{\mathbb{P}}

\newcommand{\N}{{\mathbb N}}
\newcommand{\Z}{{\mathbb Z}}

\newcommand{\F}{{\mathbb F}}

